\documentclass[a4paper, 12pt]{article}

\usepackage{amsfonts,amsmath,latexsym,color,epsfig, setspace}
\newtheorem{theorem}{Theorem}

\newenvironment{proof}
      {\medskip\noindent{\bf Proof:}\hspace{1mm}}
      {\hfill$\Box$\medskip}

\begin{document}

\title{On the Ramsey Multiplicity of Complete Graphs}
\author{David Conlon\thanks{St John's College, Cambridge, CB2 1TP,
United Kingdom. Email: {\tt D.Conlon@dpmms.cam.ac.uk}.}}
\date{}

\maketitle

\begin{abstract}
\noindent
We show that, for $n$ large, there must exist at least
\[\frac{n^t}{C^{(1+o(1))t^2}}\]
monochromatic $K_t$s in any two-colouring of the edges of $K_n$,
where $C \approx 2.18$ is an explicitly defined constant. The old
lower bound, due to Erd\H{o}s \cite{E62}, and based upon the
standard bounds for Ramsey's theorem, is
\[\frac{n^t}{4^{(1+o(1))t^2}}.\]
\end{abstract}

\section{Introduction}

Let $k_t(G)$ be the number of complete subgraphs of order $t$ in a graph
$G$, and let
\[k_t(n) = \min\{k_t(G) + k_t(\overline{G}) : |G| = n\},\]
that is, $k_t(n)$ is the minimum number of monochromatic $K_t$s within a
two-colouring of the edges of $K_n$.

Our object of interest in this paper will be the limit
\[c_t = \lim_{n \rightarrow \infty} \frac{k_t(n)}{\binom{n}{t}},\]
that is, the minimum proportion of $K_t$s in a two-colouring of
the edges of $K_n$, for $n$ large, which are monochromatic. That this
limit actually exists is a consequence of the fact that the ratios
\[c_t(n) = \frac{k_t(n)}{\binom{n}{t}}\]
are themselves increasing with $n$ (and are bounded above by 1).

Erd\H{o}s, who pioneered the study of these multiplicity constants in
\cite{E62} proved, by a simple application of Ramsey's
Theorem, that
\[c_t \geq {\binom{r(t,t)}{t}}^{-1},\]
where $r(t,t)$ is just the standard diagonal Ramsey number. He
also proved, using the probabilistic method, that
\[c_t \leq 2^{1 - \binom{t}{2}}\]
and conjectured that this is the correct value of the multiplicity
constant $c_t$.

It was later proven by Thomason \cite{T89}, by constructing specific
counterexamples, that this conjecture is false for $t \geq 4$ (that it
is true for $t=3$ follows from a result of Goodman \cite{G59}). For
example, he showed that $c_4 < 1/33$, whereas the predicted multiplicity
from Erd\H{o}s's conjecture was $1/32$.

As a consequence of Thomason's result, the question of upper bounds on
$c_t$ has received quite a lot of attention over the years, and the
bounds on many of the $c_t$ have been improved (see for example
\cite{FR93}, \cite{JST96}, \cite{FR02}). However, none of the bounds
have ever been improved to below half of the conjectured
value, so it still seems likely that the multiplicity constant is
close to the conjectured value.

It is perhaps surprising though that the lower bound has received little
attention over the years, seeing as, given the usual upper
bound for Ramsey's Theorem of $r(t,t) \leq 4^t$, the lower bound given
by Erd\H{o}s only implies that
\[c_t \geq 4^{-(1+o(1))t^2},\]
which differs vastly from the upper bound. The only other published
work relating to lower bounds is the result of Giraud \cite{Gi79} showing
that $c_4 > 1/46$. The proof of this bound is highly technical and
quite specifically tailored to the $t=4$ case, so it is quite unlikely
that it could be extended.

In this paper we will take a closer look at the problem of determining
lower bounds for these multiplicity constants, proving that
\[c_t \geq C^{-(1+o(1))t^2},\]
where $C \approx 2.18$. More precisely, we have the following
theorem:

\begin{theorem}
Let $t_{\epsilon} (x)$ be a function with $t_{\epsilon} (0) =
\epsilon$ satisfying the differential equation
\[t'_{\epsilon}(x)= \log t_{\epsilon}(x) \frac{t_{\epsilon} (x) (1
- t_{\epsilon}(x))}{x - (1+x) t_{\epsilon}(x)}.\] Let $L =
\lim_{\epsilon \rightarrow 0} t_{\epsilon} (1)$ and let $C =
(L(1-L))^{-1/2}$. Then
\[c_t \geq C^{- (1 + o(1))t^2}.\]
\end{theorem}

The proof of this theorem must, however, wait until much later. We
will instead begin by giving a surprisingly simple proof of the
bound
\[c_t \geq 2\sqrt{2}^{-(1+o(1))t^2}.\]
This is a good place to start because it requires no preliminaries
and will provide us with a proof framework to which we can refer
when we discuss how to make further improvements on the bound.

\section{A first attempt}

\begin{theorem}
Let $k, l \geq 1$ be natural numbers. Then, in any red/blue-colouring of
the edges of $K_n$, there are at least
\[2^{-k(l-2) - \binom{k+1}{2}} \binom{n}{k} - O_{k,l}(n^{k-1})\]
red $K_k$s or at least
\[2^{-l(k-2) - \binom{l+1}{2}} \binom{n}{l} - O_{k,l}(n^{l-1})\]
blue $K_l$s.
\end{theorem}

\begin{proof}
We will prove the result by induction. Note that for $k = 1$ or $l=1$
the result is trivial, since all $n$ points ($K_1$s) may be considered
as being red and blue.

Now let's assume that the theorem holds for all $(k, l)$ with
$k < k_0$ or $l < l_0$, and let's prove it for $(k_0, l_0)$. Suppose
that we have a $K_n$ the edges of which have been red/blue-coloured.
Then, for any vertex $v_i$, $i = 1, \cdots, n$, there is a colour $C_i$,
either red or blue, such that $v_i$ has at least $(n-1)/2$ neighbours
to which it is connected by this colour. Now, at least $n/2$ of these
colours $C_i$ are the same. Suppose, without loss of generality that
they are red and that the relevant vertices are $v_1, \cdots, v_{n/2}$.
Moreover, let $V_i$ be the set of red neighbours of $v_i$ for $i=1,
\cdots, n/2$. Since $|V_i| \geq (n-1)/2$, the induction hypothesis tells
us that either $V_i$ contains at least
\[2^{-(k_0 - 1)(l_0-2) - \binom{k_0}{2}} \binom{(n-1)/2}{k_0-1} -
O_{k_0,l_0}((n/2)^{k_0-2})\]
red $K_{k_0 - 1}$s or at least
\[2^{-(l_0)(k_0-3) - \binom{l_0 + 1}{2}} \binom{(n-1)/2}{l_0} -
O_{k_0,l_0}((n/2)^{l_0 - 1})\]
blue $K_{l_0}$s.

If the second case occurs for any $i \in \{1, \cdots, n/2\}$, then the
number of blue $K_{l_0}$s is at least
\begin{eqnarray*}
2^{-(l_0)(k_0-3) - \binom{l_0 + 1}{2}} 2^{-l_0} \binom{n}{l_0} -
O_{k_0,l_0}(n^{l_0 - 1}) \\
= 2^{-(l_0)(k_0-2) - \binom{l_0 + 1}{2}} \binom{n}{l_0} -
O_{k_0,l_0}(n^{l_0 - 1}).
\end{eqnarray*}
Therefore, since we are finished if this occurs, we may assume that for
every $i \in \{1, \cdots, n/2\}$, $V_i$ contains at least
\begin{eqnarray*}
2^{-(k_0 - 1)(l_0-2) - \binom{k_0}{2}} \binom{(n-1)/2}{k_0-1} -
O_{k_0,l_0}((n/2)^{k_0-2}) \\
= 2^{-(k_0 - 1)(l_0-2) - \binom{k_0+1}{2} +
1} \binom{n}{k_0-1} - O_{k_0,l_0}(n^{k_0-2})
\end{eqnarray*}
red $K_{k_0 - 1}$s. But each such red $K_{k_0-1}$ in $V_i$ is connected
all in red to the vertex $v_i$. Seeing as we may potentially overcount
$k_0$ times, we see that the number of monochromatic $K_{k_o}$s is
therefore at least
\begin{eqnarray*}
\frac{1}{k_0} \frac{n}{2} \left(2^{-(k_0 - 1)(l_0-2) -
\binom{k_0+1}{2} + 1} \binom{n}{k_0-1} - O_{k_0,l_0}(n^{k_0-2})\right)
\\
= 2^{-(k_0 - 1)(l_0-2) -\binom{k_0+1}{2}} \binom{n}{k_0} -
O_{k_0,l_0}(n^{k_0-1}),
\end{eqnarray*}
which implies the required bound. The result follows similarly if more
of the $C_i$ are blue than red.
\end{proof}

Our claim about $c_t$, that
\[c_t \geq 2\sqrt{2}^{-(1+o(1))t^2},\]
now follows from taking $k=l=t$.

\section{Balancing the argument}

We are now going to smooth out the argument of the previous
section by shifting around the thresholds that determine whether
we are considering the red or blue neighbourhood of a given
vertex. To make this more explicit, we will need two
definitions.\\

{\bf Definition} {\it Consider the lattice $L$ of points $(i,j)$
such that $i,j \geq 1$. A path in $L$ to $(k,l)$ is defined to be
a sequence $P = \{(a_i, b_i)\}_{i=0,\cdots,m}$ such that
\begin{itemize}
\item
$a_0 = 1$ or $b_0=1$,
\item
$a_1 \neq 1$ and $b_1 \neq 1$,
\item
$(a_i, b_i) = (a_{i-1}+1, b_{i-1})$ or $(a_{i-1}, b_{i-1}+1)$, and
\item
$(a_m,b_m) = (k,l)$.
\end{itemize}
We will denote the set of all paths to $(k,l)$ by $\Pi_{k,l}$.}\\

{\bf Definition} {\it A threshold sequence $\{t_{i,j}\}_{i,j \geq
2}$ is a sequence of real numbers such that $0 < t_{i,j} < 1$ and
$t_{j,i} = 1 - t_{i,j}$.} \\

Given these definitions, we may now state and prove the following
more general version of Theorem 2:

\begin{theorem}
Let $k, l \geq 1$ be natural numbers, and let $\{t_{i,j}\}_{i,j
\geq 2}$ be a threshold sequence. Given a path $P =
\{(a_i,b_i)\}_{i=0, \cdots, m}$, define $\{s_i\}_{i=1, \cdots, m}$
by
\[s_i = \left\{ \begin{array}{ll}
        t_{a_i,b_i} & \mbox{if $(a_i,b_i) = (a_{i-1}, b_{i-1}+1)$};\\
        1 - t_{a_i,b_i} & \mbox{if $(a_i,b_i) = (a_{i-1}+1, b_{i-1})$},
\end{array} \right.\]
and let $U_{k,l}$ and $V_{k,l}$ be given by
\[U_{k,l} = \min_{P \in \Pi_{k,l}} \prod_{i=1}^{m} s_i^{a_i}, V_{k,l} = \min_{P \in \Pi_{k,l}} \prod_{i=1}^{m} s_i^{b_i}.\]
Then, in any red/blue-colouring of the edges of
$K_n$, there are at least
\[U_{k,l} \binom{n}{k} - O_{k,l}(n^{k-1})\]
red $K_k$s or at least
\[V_{k,l} \binom{n}{l} - O_{k,l}(n^{l-1})\]
blue $K_l$s.
\end{theorem}

\begin{proof}
The proof is quite similar to the proof of Theorem 2. For $k=1$ or
$l=1$, the theorem is trivial, since a product over an empty set
is simply 1. We may therefore assume that we know the theorem
holds for all $(k,l)$ such that $k < k_0$ or $l < l_0$ and attempt
to prove it at $(k_0, l_0)$.

Suppose we have a red/blue-colouring of the edges of $K_n$. Then every
vertex $v$ is connected to either $t_{k_0,l_0} (n-1)$ vertices by blue
edges or $(1 - t_{k_0,l_0}) (n-1)$ vertices by red edges. Moreover,
either the first case occurs $t_{k_0,l_0} n$ times or the second case
occurs $(1 - t_{k_0,l_0}) n$ times. Let us assume without loss of
generality that the case where we have $t_{k_0,l_0} n$ vertices
$v_i$, each of which is connected to $t_{k_0,l_0} (n-1)$ neighbours by
blue edges, occurs. For each $v_i$ we may conclude, by using the
induction hypothesis, that the set $V_i$ of blue neighbours contains
either
\[U_{k_0,l_0-1} \binom{t_{k_0,l_0} n}{k_0} - 0_{k_0,l_0} (n^{k_0-1})\]
red $K_{k_0}$s or
\[V_{k_0,l_0-1} \binom{t_{k_0,l_0} n}{l_0 - 1} - 0_{k_0,l_0}
(n^{l_0-2})\]
blue $K_{l_0 - 1}$s.

If the first case occurs for any of the vertices we are done, since
\begin{eqnarray*}
U_{k_0,l_0-1} \binom{t_{k_0,l_0} n}{k_0} - 0_{k_0,l_0} (n^{k_0-1}) &=&
t_{k_0,l_0}^{k_0} U_{k_0,l_0-1} \binom{n}{k_0} - 0_{k_0,l_0}
(n^{k_0-1})\\
& \geq & U_{k_0,l_0} \binom{n}{k_0} - 0_{k_0,l_0} (n^{k_0-1}).
\end{eqnarray*}
Therefore we must have that the second case occurs for all of the $v_i$.
We may then conclude, since each of the $V_i$ is connected to $v_i$ by
all blue edges, that we have
\begin{eqnarray*}
\frac{1}{l_0} t_{k_0,l_0} n \left(V_{k_0,l_0-1} \binom{t_{k_0,l_0}
n}{l_0 - 1} - 0_{k_0,l_0} (n^{l_0-2})\right) & = & t_{k_0,l_0}^{l_0}
V_{k_0,l_0-1} \binom{n}{l} - 0_{k_0,l_0} (n^{l_0-1}) \\
& \geq & V_{k_0,l_0} \binom{n}{l_0} - 0_{k_0,l_0} (n^{l_0-1})
\end{eqnarray*}
blue $K_l$s, so we are done in this case as well.

The result follows similarly in the case where we have $(1-t_{k_0,l_0})
n$ vertices with red degree greater than or equal to $(1-t_{k_0,l_0})
(n-1)$.
\end{proof}

Our first result about $c_t$, that
\[c_t \geq 2\sqrt{2}^{-(1+o(1))t^2}\]
follows again from this result by taking $t_{i,j} = 1/2$ everywhere and
noting that a path to $(t,t)$ which minimises
\[\prod_{i=1}^{m} s_i^{\max(a_i,b_i)} = \prod_{i=1}^{m}
2^{-\max(a_i,b_i)}\]
is $(2,1), (2,2), (3,2), \cdots, (t,2), (t,3), (t,4), \cdots, (t,t)$.

The rest of the paper will be concerned with replacing this somewhat
artificial minimum with a more natural one. To do so we will choose a
threshold sequence in such a way that the products
\[\prod_{i=1}^{m} s_i^{\max(a_i,b_i)}\]
are the same (or very nearly the same) for all paths to $(t,t)$. It is
possible to encapsulate the result of making such a judicious choice of
threshold sequence within a recurrence relation. This will be our next
theorem.

\begin{theorem}
Let $M_{k,l}$ be the sequence given by the following conditions:
\begin{itemize}
\item
$M_{k,1} = M_{1, l} = 1$, for all $k,l \geq 1$;
\item
$M_{k,l} = (M_{k,l-1}^{-1/\mu} + M_{k-1,l}^{-1/\mu})^{-\mu}$,
where $\mu = \max(k, l)$.
\end{itemize}
Then
$c_t \geq M_{t,t}$.
\end{theorem}

\begin{proof}
Given a threshold sequence $\{t_{i,j}\}_{i,j \geq 2}$ and a path $P =
\{(a_i,b_i)\}_{i=0, \cdots, m} \in \Pi_{k,l}$, define $\{s_i\}_{i=1,
\cdots, m}$ by
\[s_i = \left\{ \begin{array}{ll}
        t_{a_i,b_i} & \mbox{if $(a_i,b_i) = (a_{i-1}, b_{i-1}+1)$};\\
        1 - t_{a_i,b_i} & \mbox{if $(a_i,b_i) = (a_{i-1}+1, b_{i-1})$},
\end{array} \right.\]
Then we know, from Theorem 3, that
\[c_t \geq S_{t,t},\]
where
\[S_{k,l} = \min_{P \in \Pi_{k,l}} \prod_{i=1}^{m} s_i^{\max(a_i,b_i)}.\]

Now, suppose that we have chosen the values of a threshold
sequence $\{t_{i,j}\}_{i,j \geq 2}$ for all $(i,j)$ such that $i <
k$ or $j < l$. Then the optimal way to choose $t_{k,l}$ so as to
maximise $S_{k,l}$ will be to take $t_{k,l}$ satisfying
\[t_{k,l}^{\mu} S_{k,l-1} = (1 - t_{k,l})^{\mu} S_{k-1,l},\]
that is
\[t_{k,l} = \frac{\sqrt[\mu]{S_{k-1,l}}}{\sqrt[\mu]{S_{k,l-1}} + \sqrt[\mu]{S_{k-1,l}}}.\]
Then, since
\[S_{k,l} = \min(t_{k,l}^{\mu} S_{k,l-1}, (1 - t_{k,l})^{\mu} S_{k-1,l}),\]
we have that
\[S_{k,l} = (S_{k,l-1}^{-1/\mu} + S_{k-1,l}^{-1/\mu})^{-\mu}.\]
If we had chosen this optimal threshold at each step it is easy to see
(since $S_{k,1} = S_{1,l} = 1$ for all $k,l \geq 1$) that
\[S_{k,l} = M_{k,l},\]
so, in particular,
\[c_t \geq S_{t,t} = M_{t,t}.\]
\end{proof}

Unfortunately, though this recurrence does, in theory, provide the solution to
the problem of finding the best possible threshold sequence, the recurrence appears
to be very difficult to solve explicitly. A computation done up to relatively large
values of $t$ does however suggest that
\[c_t \geq C^{-(1+o(1))t^2},\]
where $C$ is approximately $2.18$. In the next section, we shall analytise the
problem in order to show that this is indeed the case.

\section{The analytic approach}

As usual, given a threshold set $\{t_{i,j}\}_{i,j \geq 2}$ and a path $P =
\{(a_i,b_i)\}_{i=0, \cdots, m} \in \Pi_{t,t}$, we define $\{s_i\}_{i=1,
\cdots, m}$ to be
\[s_i = \left\{ \begin{array}{ll}
        t_{a_i,b_i} & \mbox{if $(a_i,b_i) = (a_{i-1}, b_{i-1}+1)$};\\
        1 - t_{a_i,b_i} & \mbox{if $(a_i,b_i) = (a_{i-1}+1, b_{i-1})$},
\end{array} \right.\]
Theorem 3 then tells us that
\[c_t \geq S_{t,t} = \min_{P \in \Pi_{t,t}} S_P,\]
where
\[S_P = \prod_{i=1}^{m} s_i^{\max(a_i,b_i)}.\]
Note that because any threshold set is chosen to be symmetric it is
sufficient to take the minimum over all paths which lie strictly under
the diagonal.

Now, let us suppose that our function $t_{k,l}$ is a smooth function of
the form $t_{k,l} = t(l/k)$. Let $P_1$ and $P_2$ be two paths which
differ by very little. More explicitly, suppose that the two paths are
the same, except that between $(a-1,b-1)$ and $(a,b)$, $P_1$ follows
the path $(a-1,b-1), (a-1, b), (a, b)$ and $P_2$ follows the alternate path
$(a-1,b-1), (a, b-1), (a, b)$. Let's calculate the ratio
$S_{P_1}/S_{P_2}$. Since the paths agree everywhere except between
$(a-1,b-1)$ and $(a,b)$, this ratio is, for $a > b$, simply
\[\frac{(1-t_{a,b})^a (t_{a-1,b})^{a-1}}{(t_{a,b})^a (1 - t_{a,
b-1})^a}.\]
But now, since $t_{k,l} = t(l/k)$ is presumed smooth, we may write
\[t_{a-1,b} = t(b/(a-1)) = t(b/a) + t'(b/a) \frac{b}{a(a-1)} \pm
\frac{M}{a^2}\]
and
\[t_{a, b-1} = t((b-1)/a) =  t(b/a) - t'(b/a) \frac{1}{a} \pm
\frac{M}{a^2},\] where $M$ may be taken to be the maximum value of
$|t''(x)|$ over the range $0 \leq x \leq 1$.

If we write out and simplify the ratio using these expansions, we get
\[\frac{\left(1 + \frac{t'(b/a)}{t(b/a)} \frac{b}{a(a-1)} \pm
\frac{M}{t(b/a) a^2}\right)^{a-1}}{t(b/a) \left( 1 +
\frac{t'(b/a)}{1 - t(b/a)}\frac{1}{a} \pm \frac{M}{(1 -
t(b/a))a^2}\right)^a}.\] Using the approximation, valid for $|z|$
sufficiently small, that $|e^z - 1 - z| \leq z^2$, we see that for
$a \geq a_0$, say, the ratio is given by
\[R_{a, b} = \exp\left\{\frac{t'(b/a)}{t(b/a)} \frac{b}{a} - \log
t(b/a) - \frac{t'(b/a)}{1 - t(b/a)} \pm \frac{M'}{a}\right\},\]
where $M'$ and $a_0$ are constants which depend only on $t(x)$.
Note that $M'$ and $a_0$ have at most a polynomial dependence on
the maximum values of $1/t(x), 1/(1- t(x)), |t'(x)|$ and
$|t''(x)|$ in the range $0 \leq x \leq 1$.

We now choose $t_{\epsilon} (x): [0,1] \rightarrow \mathbb{R}$ to
be a smooth function satisfying:
\begin{itemize}
\item
$t_{\epsilon} (0) = \epsilon$,
\item
$t'_{\epsilon}(x)= \log t_{\epsilon}(x) \frac{t_{\epsilon} (x) (1
- t_{\epsilon}(x))}{x - (1+x) t_{\epsilon}(x)}$.
\end{itemize}
That, for $\epsilon > 0$, such a function exists is a consequence
of standard existence results in the theory of differential
equations (the Picard-Lindel\"of Theorem, for example). Note that
each $t_{\epsilon} (x)$ is increasing. Moreover, it is easily
verified that the function $t_1 (x) = 1$ everywhere, so that, for
$\epsilon < 1$, $t_{\epsilon} (x)$ is always less than $1$. We
shall assume that it is always bounded away from $1$ by
$\epsilon'$.

Let $\delta = \min(\epsilon, \epsilon')$. It is easy to verify
that the maximum values of $|t'_{\epsilon}(x)|$ and
$|t''_{\epsilon}(x)|$ have at most a polynomial dependence on
$\delta$. Now let's turn our attention to estimating that part of
the exponent of $R_{i,j}$ given by
\[E'_{i,j} = \frac{t'_{\epsilon}(i/j)}{t_{\epsilon}(i/j)} \frac{i}{j} - \log
t_{\epsilon}(i/j) - \frac{t'_{\epsilon}(i/j)}{1 -
t_{\epsilon}(i/j)}.\] We are going to approximate this by the
corresponding integral
\[I_{i,j} = j \int_{(i-1)/j}^{i/j} \left(\frac{t'_{\epsilon}(x)}{t_{\epsilon}(x)} x - \log
t_{\epsilon}(x) - \frac{t'_{\epsilon}(x)}{1 -
t_{\epsilon}(x)}\right) \,dx.\] To show that this is a good
approximation, note that the function
\[g(x) = \frac{t'_{\epsilon}(x)}{t_{\epsilon}(x)} x - \log t_{\epsilon}(x) - \frac{t'_{\epsilon}(x)}{1 -
t_{\epsilon}(x)}\] may be approximated, from the mean-value
theorem, by
\[|g(x + \eta) - g(x)| \leq \max_{0 \leq \theta \leq \eta} (g'(x +
\theta)) \eta.\] But now, again, since the maximum values of
$|t'_{\epsilon}(x)|$ and $|t''_{\epsilon}(x)|$ have a polynomial
dependence on $\delta$, and the first derivative of $g$ may be
expressed in terms of these, we see that
\[|g(x + \eta) - g(x)| \leq \delta^{-c} \eta\]
for some constant $c$.

Assume, without loss of generality, that $a_0$ was chosen to be
greater than or equal to $\delta^{-(c+1)}$. It may need to be
larger still, as we will later need $a_0$ to be larger than the
maximum value of both $|t'_{\epsilon} (x)|$ and $|t''_{\epsilon}
(x)|$, but this is certainly achieved by taking it to be some
fairly large power of $\delta^{-1}$. We therefore see that, for $j
\geq a_0$,
\begin{eqnarray*}
\left|I_{i,j} - E'_{i,j}\right| & = & \left|j \int_{(i-1)/j}^{i/j}
(g(x) - g(i/j)) \,dx\right|\\ & \leq & \int_{(i-1)/j}^{i/j} j
|g(x) - g(i/j)| \,dx
\\ & \leq & \frac{1}{\delta^c j}\\ & \leq & \delta.
\end{eqnarray*}

But now, by the definition of $t_{\epsilon}(x)$,
\[I_{i,j} = 0,\]
so that
\[|E'_{i,j}| \leq \delta.\]
Now the exponent $E_{i,j}$ of $R_{i,j}$ is
\[E_{i,j} = E'_{i,j} \pm M'/j,\]
and therefore, since the constant $M'$ has at most a polynomial
dependence on $\delta$,
\[|E_{i,j}| \leq \delta + \frac{1}{\delta^d j},\]
for some constant $d$.

Therefore, if $t_{\epsilon}$ were our threshold function, we see
that the maximum difference between two different paths leading to
$(t,t)$ would be a factor of
\[R = \prod_{1 \leq j \leq t} \prod_{1 \leq i \leq j} \exp\{|E_{i,j}|\}.\]
The exponent of this function is
\[E = \sum_{1 \leq j \leq t} \sum_{1 \leq i \leq j} |E_{i,j}|\]
which by our earlier estimate on $E_{i,j}$ for $j \geq a_0$ is
less than
\begin{eqnarray*}
\sum_{a_0 \leq j \leq t} \sum_{1 \leq i \leq j} |E_{i,j}| +
\delta^{-e} & \leq & \sum_{a_0 \leq j \leq t} \sum_{1 \leq i \leq
j} \left(\delta + \frac{1}{\delta^d j}\right) + \delta^{-e}\\ &
\leq & \delta t^2 + \delta^{-d} t + \delta^{-e},
\end{eqnarray*}
the $\delta^{-e}$ term coming from the terms with $j < a_0$.

Now the value of $S_D$ along the diagonal path $D$ given by
\[(2,1), (2,2), (3,2), (3,3), \cdots, (t,t-1), (t,t)\]
is simply
\[(t_{\epsilon} (1) (1 - t_{\epsilon} (1)))^{\binom{t}{2}} = C_{\epsilon}^{- t^2},\]
say. Therefore, we see that for all paths $P$
\[S_P \geq C_{\epsilon}^{-t^2} \times \exp\left\{-(\delta t^2 + \delta^{-d} t
+ \delta^{-e})\right\}.\] Taking limits in an appropriate fashion
should now imply Theorem 1, which, we recall, states

{\bf Theorem 1} {\it Let $t_{\epsilon} (x)$ be a function with
$t_{\epsilon} (0) = \epsilon$ satisfying the differential equation
\[t'_{\epsilon}(x)= \log t_{\epsilon}(x) \frac{t_{\epsilon} (x) (1
- t_{\epsilon}(x))}{x - (1+x) t_{\epsilon}(x)}.\] Let $L =
\lim_{\epsilon \rightarrow 0} t_{\epsilon} (1)$ and let $C =
(L(1-L))^{-1/2}$. Then
\[c_t \geq C^{- (1 + o(1))t^2}.\]}

This theorem agrees beautifully with the computer calculation from
the previous section, but is not yet rigorous, for the simple
reason that $t_{\epsilon}$ does not give a well-defined threshold
sequence. The problem is that it is not symmetric at the diagonal.
In fact, calculations suggest that the value of $t_{\epsilon} (1)$
for $\epsilon$ small is closer to $0.7$ than $0.5$.

We're nearly there though, as we can define a well-behaved
threshold function $T_{k,l}^{\epsilon}$ using $t_{\epsilon}$. For
$l < k - \delta^{-1}$ (or rather the integer value thereof), this
threshold sequence is just $t_{\epsilon} (l/k)$, but near the
diagonal we have to do something a little different.

To define $T_{k,l}^{\epsilon}$ in the range $l \geq k -
\delta^{-1}$ we will define numbers $a_0, \cdots, a_{\delta^{-1}}$
such that $T_{k, k - m}^{\epsilon} = a_m$ for each $m = 0, \cdots,
\delta^{-1}$. Naturally, we take $a_0$ to be just $1/2$. Suppose
now that we have defined $a_0, \cdots, a_{i-1}$ and that we would
like to define $a_i$. Like earlier, we want the term
\[\frac{(1-T_{k,k-i})^k (T_{k-1,k-i})^{k-1}}{(T_{k,k-i})^k (1 - T_{k,
k-i-1})^k} = \frac{(1-a_{i})^k a_{i-1}^{k-1}}{a_{i}^k
(1-a_{i+1})^k}\] to be as close to one as possible. We will not
split hairs, however, and simply define $a_{i+1}$ by the relation
\[\frac{1 - a_{i+1}}{1 - a_i} = \frac{a_{i-1}}{a_i},\]
so that the ratio above simply takes the value $1/a_{i-1}$. Note
also that, provided $a_1 \geq 1/2$, the $a_i$ are strictly
increasing, so that in particular $1/a_{i-1}$ is always smaller
than $2$. But now, for any given initial value $a_1 \geq 1/2$,
denote by $f(a_1)$ the value $a_{\delta^{-1}}$. A fairly
straightforward observation is that $f$ is a continuous strictly
increasing function with respect to $a_1$, with $f(1/2) = 1/2$ and
$f(1) = 1$, so in particular, there exists $a_{\epsilon}$ such
that $f(a_{\epsilon}) = t_{\epsilon} (1)$.

Now it is easily verified that
\[\frac{a_i}{a_{i+1}} \geq \frac{1-a_{i+1}}{1-a_i} =
\frac{a_{i-1}}{a_i},\] so that the ratios are successively
decreasing. Therefore,
\[\left(\frac{a_m}{a_{m-1}}\right)^m \leq \frac{a_{m} \cdots
a_1}{a_{m-1} \cdots a_0} \leq 2,\] so that, for $\delta$
sufficiently small,
\[\frac{a_{\delta^{-1}}}{a_{\delta^{-1}-1}} \leq 2^{\delta} \leq 1 + \delta.\]

Now choose $a_1 = a_{\epsilon}$, and define the threshold sequence
$T_{k,l}^{\epsilon}$ using this as the seed. We already know that
any path $P$ which lies below the path $Q$ given by
\[(\delta^{-1} + 3, 1), (\delta^{-1} + 3, 2), (\delta^{-1} + 4, 2), \cdots,
(t, t - \delta^{-1} - 2), (t, t - \delta^{-1} - 1)\] satisfies the
inequality
\[S_P \geq C_{\epsilon}^{-t^2} \times \exp\left\{-(\delta t^2 + \delta^{-d} t
+ \delta^{-e})\right\}.\] How much more error do we incur pushing
this up to the diagonal?

Again, suppose that $P_1$ and $P_2$ are paths that agree
everywhere, except that $P_1$ follows the path $(a-1,b-1), (a-1,
b), (a, b)$ and $P_2$ follows the alternate path $(a-1,b-1), (a,
b-1), (a, b)$. For $b \geq a - \delta^{-1} + 1$, all of the
entries are defined so that the error coming from such a
distortion is at most a factor of $2$. Therefore, the total amount
of extra error we could get from distorting paths within this
region is just
\[2^{\delta^{-1} t}.\]
The only problems might be when $b = a - \delta^{-1}$ or $a -
\delta^{-1} - 1$, these being the cases where the two different
pieces of our threshold sequence mix. In the first case we have to
consider the ratio
\[R_{a, a - \delta^{-1}} = \frac{(1 - a_{\delta^{-1}})^a (a_{\delta^{-1} -
1})^{a-1}}{(a_{\delta^{-1}})^a (1 - t_{\epsilon} (1 -
\frac{\delta^{-1} + 1}{a}))^a}.\] We may use the facts that
$a_{\delta^{-1}} = a_{\delta^{-1} - 1} \pm \delta$ and
\[t_{\epsilon} \left( 1 - \frac{\delta^{-1} + 1}{a}\right) = t_{\epsilon} (1)
\pm 2 \delta = a_{\delta^{-1}} \pm 2 \delta\] (recall that we took
$a_0$ earlier to be a sufficiently large power of $\delta^{-1}$),
to show that, for $\delta$ small,
\[|\log R_{a, a - \delta^{-1}}| \leq 10 \delta t\]
A similar result holds when $b = a - \delta^{-1} - 1$, so the
total extra error incurred from all these crossover squares is at
most $e^{20 \delta t^2}$. We therefore see, finally, that for all
paths $P$
\[S_P \geq C_{\epsilon}^{-t^2} \times \exp\left\{-(21 \delta t^2 + (\delta^{-d} + \delta^{-1}) t
+ \delta^{-e})\right\}.\] Taking limits appropriately now implies
Theorem 1.

\section{Concluding remarks}

With the proof now finished, we wish to point out that we believe
Theorem 3 to be the natural analogue, for multiplicities, of the
well-known Erd\H{o}s-Szekeres argument \cite{ES35} for proving
Ramsey's Theorem. To bring out the correspondence note that the
result implied by the standard proof of Ramsey's Theorem may be
written as\\

{\bf Ramsey's Theorem}
{\it Let $k, l \geq 1$ be natural numbers, and let $\{t_{i,j}\}_{i,j
\geq 2}$ be a threshold set. Given a path $P = \{(a_i,b_i)\}_{i=0,
\cdots, m}$, define $\{s_i\}_{i=1, \cdots, m}$ by
\[s_i = \left\{ \begin{array}{ll}
        t_{a_i,b_i} & \mbox{if $(a_i,b_i) = (a_{i-1}, b_{i-1}+1)$};\\
        1 - t_{a_i,b_i} & \mbox{if $(a_i,b_i) = (a_{i-1}+1, b_{i-1})$},
\end{array} \right.\]
and let $n = n_{k,l}$ be given by
\[n_{k,l} = \min_{P \in \Pi_{k,l}} \prod_{i=1}^{m} \frac{1}{s_i}.\]
Then, in any red/blue-colouring of the edges of $K_n$ there is
either a red $K_k$ or a blue $K_l$. }\\

If we were to take the threshold set to be $1/2$ everywhere this yields
\[r(k,l) \leq 2^{k+l-3},\]
but, if instead we follow Erd\H{o}s and Szekeres by choosing $t_{k,l} =
\frac{l}{k+l}$ (which is the correct choice to make the product over
each path the same), we get the bound
\[r(k,l) \leq \binom{k+l}{k}.\]
Near the diagonal, this, disappointingly, does little better than the
easy bound coming from taking a trivial threshold set. This we feel can
be explained by the fact that almost all paths do lie close to the
diagonal, a fact which is quite to our advantage in the multiplicities
case.

Despite the large improvement that we have made in this paper, it
still seems likely that the true values of the multiplicity
constants are much smaller still, perhaps even as small as
\[\sqrt{2}^{-(1+o(1))t^2}.\]
While we will not be so bold as to conjecture such a bound, we do
conjecture that there exists a constant $\epsilon > 0$ such that
\[c_t \geq (C-\epsilon)^{-(1+o(1))t^2},\]
though we feel that making such an improvement would be of comparable
difficulty to showing that
\[r(t,t) \leq (4-\epsilon)^t.\]

It is also worth noting that the methods of this paper generalise
easily to cover the multicolour case, that is to count the number
of monochromatic $t$-cliques within a $q$-colouring. The analogue
of Theorem 4, for example, reads

\begin{theorem}
Let $M_{k_1, \cdots, k_q}$ be the sequence defined by the
following conditions:
\begin{itemize}
\item
$M_{k_1, \cdots, 1} = \cdots = M_{1, cdots, k_q} = 1$, for all
$k_1, \cdots, k_q \geq 1$;
\item
$M_{k_1, k_2, \cdots, k_q} = (M_{k_1 - 1, k_2, \cdots,
k_q}^{-1/\mu} + \cdots + M_{k_1, k_2, \cdots, k_q - 1}^{-1/\mu}
)^{-\mu}$, where $\mu = \max(k_1, \cdots, k_q)$.
\end{itemize}
Then the number of monochromatic $K_t$ within any $q$-colouring of
$K_n$ is at least
\[M_{t, t, \cdots, t} \binom{n}{t} + C_{q, t} n^{t-1}.\]
\end{theorem}

We have not worked out the multicolour analogue of Theorem 1 but
doubtless it can also be generalised. It is unclear as to how
enlightening this would be though, as, like Theorem 5 above, it
would probably still require a computer test to estimate the value
of the constant for any particular $q$. It would certainly be
interesting, however, if one could accurately determine the way in
which the constants $M_{t, t, \cdots, t}$ grow with $q$. Is there,
for example, a constant $\alpha < 1$ such that
\[M_{t,t, \cdots, t} \geq q^{-\alpha q t^2},\]
for all $q$?

\end{document}